\documentclass[a4paper,12pt]{amsart}

\usepackage{amssymb,amsmath,latexsym,amsthm}
\usepackage{datetime}
\usepackage{xcolor}
\usepackage{tikz}
\usepackage{float}
\usetikzlibrary{calc,arrows.meta}

\usepackage[pdfauthor={Dimitrios Chatzakos},
pdftitle={},
            pdfkeywords={L-functions},
             pdfcreator={Pdflatex}]
{hyperref}
\hypersetup{colorlinks=true}
\usepackage{xcolor}
\usepackage{enumerate}

\usepackage[english]{babel}

\theoremstyle{plain}
\newtheorem{theorem}{Theorem}[section]
\newtheorem{thm}[theorem]{Theorem}
\newtheorem{lemma}[theorem]{Lemma}

\newtheorem{proposition}[theorem]{Proposition}

\newtheorem*{rem*}{Remark}

\theoremstyle{definition}

\newtheorem{remark}[theorem]{Remark}

\def\MR#1{\href{http://www.ams.org/mathscinet-getitem?mr=#1}{MR#1}}

\def\eps{\epsilon}

\def \e {{\varepsilon}}

\def \R {{\mathbb R}}

\def \C {{\mathbb C}}
\def \g {\gamma}
\def \G {\Gamma}

\def \GmodH {{\Gamma\backslash\mathbb H}}

\def \supp {{\rm supp\,} }

\DeclareMathOperator{\arccosh}{arcosh}

\newcommand{\abs}[1]{\left\lvert #1 \right\rvert}

\title[Square root cancellation in the hyperbolic lattice counting problem]{Square root cancellation in the hyperbolic lattice counting problem over cocompact groups}

\author[Chatzakos]{Dimitrios Chatzakos}
\address{
   Department of Mathematics,
   University of Patras,
   26 504, Patras,
   Greece
}
\email{dchatzakos@math.upatras.gr}

\author[Dimakis]{Panagiotis Dimakis}
\address{
Department of Mathematics, 
University of Maryland,
College Park 20740, MD, 
USA.
}
\email{pdimakis12345@gmail.com}

\date{\today}


\begin{document}

\begin{abstract}
For cocompact Fuchsian groups $\Gamma \leq \hbox{PSL}(2, \mathbb{R})$ we improve Selberg's bound for the error term of the counting function in the hyperbolic lattice counting problem, achieving an essentially optimal upper bound for the error term 
\begin{equation*} 
  E(X;z, w)  = O(X^{1/2+\varepsilon})
\end{equation*}
for almost every pair of points $z,w$. We extend this pointwise result for cocompact lattices acting on the $n$-dimensional real hyperbolic space $\mathbb{H}^n$. Moreover, in the $2$-dimensional case we study the second moment of the error term of the counting function. We improve the upper bounds of Chamizo and Cherubini for almost all pairs $z,w$, but we disprove a conjecture of Phillips and Rudnick for the case of cocompact Fuchsian groups.
\end{abstract}

\maketitle

\section{Introduction}

\subsection{Historical background}

Let $\Gamma$ be a cofinite Fuchsian group acting on the hyperbolic plane $\mathbb{H}^2$, $z$ and $w$ two fixed points on $\mathbb{H}^2$ and $\rho(z,w)$ their hyperbolic distance. Let us also denote by $u(z,w)$ the point-pair invariant function
\begin{equation} \label{pointpairinvariant}
  u(z,w) = \frac{|z-w|^2}{4 \Im(z) \Im(w)},
\end{equation}
which is explicitly related to the hyperbolic distance by the relation $\cosh \rho (z,w) = 2 u(z,w) +1$. The classical hyperbolic lattice counting problem asks to estimate the asymptotic growth of the counting function
\begin{equation*}
  N(X;z, w)  = \# \{ \gamma \in \Gamma : 4u( z, \gamma w) + 2 \leq X \}
\end{equation*}
as $X \to +\infty$. The main theorem, first proved by Selberg \cite{Selberg:1977},
states that for any cofinite group the counting function satisfies the asymptotic growth
\begin{equation*}
N(X;z, w)  = M(X;z,w) + E(X;z,w),
\end{equation*}  
where the main term $M(X;z,w)$ is a well understood quantity given as a finite sum over the small eigenvalues $\lambda_j \leq 1/4$ of the hyperbolic Laplacian and the error term $E(X;z,w)$ satisfies the upper bound
\begin{equation} \label{selbergbound}
E(X;z,w) = O(X^{2/3})
\end{equation}
(here the implied constant may depend only on $z,w,$ and $\Gamma$). We refer to \cite{Good:1983, Iwaniec:2002, Patterson:1975} for more details on this problem and its history.

The main conjecture in the hyperbolic lattice counting problem predicts that for any group $\Gamma$, for any pair of points $z,w$ and for any $\varepsilon>0$ the bound
\begin{equation} \label{mainconjecture}
  E(X;z, w)  = O(X^{1/2+\varepsilon})
\end{equation}
holds. This conjecture was first supported by second moment estimates of Chamizo \cite{Chamizo:1996a} and by lower bounds and mean value results of Phillips and Rudnick \cite{PhillipsRudnick:1994}. We also refer to \cite{Chatzakos:2017, Cherubini:2018, CherubiniRisager:2018, PetridisRisager:2018a} for further results supporting the validity of the conjecture \eqref{mainconjecture}. However, for more than fifty years the upper bound $O (X^{2/3})$ of Selberg for the error term had not been improved for any cofinite group $\Gamma$ and for any pair of points $z, w$. 

In 2025, Chatzakos--Cherubini--Lester--Risager \cite{cclr} proved the first unconditional pointwise improvement of Selberg's bound for $\Gamma = \hbox{PSL}(2,\mathbb{Z})$ and for the case where $z,w$ are Heegner points of different discriminants. Their main result in \cite{cclr} was the following:
\begin{thm} \label{cclrmaintheorem}
  Let $\Gamma$ be the modular group and $z_d,z_{d'}\in \mathbb{H}^2$ be two Heegner points of different negative squarefree discriminants $d,d'$ respectively. Then
  \begin{equation*}
    E(X;z_d, z_{d'})  = O_{d,d'} \left(\frac{X^{2/3}}{(\log X)^{1/6} }\right).
  \end{equation*}
\end{thm}

The proof of this result is arithmetic in nature, using Waldspurger's formula and the theory of fractional moments of  $L$-functions.

\subsection{Our main result}

In this paper, we prove the following new unconditional pointwise improvement of Selberg's bound for any cocompact Fuchsian group $\Gamma$ and for almost all pairs of points $z,w \in \mathbb{H}^2$. Our result goes all the way down to the conjectural square-root cancellation. 

\begin{theorem} \label{ourmainthm1}
Let $\Gamma$ be a cocompact Fuchsian group and fix a point $w \in \mathbb{H}^2$. Then, there exists a set $B_w \subset \Gamma \backslash \mathbb{H}^2$ of full hyperbolic measure, such that for every $z \in B_w$ and for every $\epsilon >0 $ the bound
\begin{eqnarray*}
E(X;z,w) = O_{\Gamma, w,z,\epsilon}  \left(X^{1/2} (\log X)^{3/2} (\log \log X)^{1/2+\epsilon}\right)
\end{eqnarray*}
holds as $X \to \infty$, uniformly for all sufficiently large $X$. In particular, the bound \eqref{mainconjecture} holds for almost all pairs of points. 
\end{theorem}

In the next subsection we discuss the main ideas in the proof of Theorem \ref{ourmainthm1}.

\subsection{Spectral theory and comparison of the methods of proofs}

We now give a brief explanation of our strategy and compare it to what was previously known. Let $\{ \phi_j \}_{j=0}^{\infty}$ be an orthonormal basis for the discrete spectrum of the positive automorphic Laplacian of the Riemann surface $\Gamma \backslash \mathbb{H}^2$, with corresponding eigenvalues $ \{\lambda_j \}_{j=0}^{\infty}$. Set $\lambda_j =  1/4 +t_j^2$. For the rest of this paper we also set $X = 2 \cosh R \asymp e^R$, i.e. $R = \cosh^{-1}(X/2) \asymp \log X$. 

Roughly speaking, after applying the pre-trace formula to a sufficient automorphic kernel (see Section \ref{section2}), bounding the error term in the hyperbolic circle problem reduces to bounding the spectral exponential sums
\begin{equation} \label{BTXdef}
\mathfrak{B} (T, X;z,w) := \mathfrak{B} (T,X) = \sum_{T \leq t_j < 2T} h (t_j) \phi_j(z) \overline{\phi_j(w)},   
\end{equation}
where $h (t_j) = h_R(t_j)$ is the spherical transform of the chosen automorphic kernel. Approximations of the spherical transform reduce the problem to estimates for the spectral exponential sum
\begin{equation}\label{hlfwvkrnl}
  \mathfrak{S} (T,X;z,w) := \mathfrak{S} (T,X) =  \sum_{T \leq t_j < 2T} X^{it_j} \phi_j(z) \overline{\phi_j(w)}.
\end{equation}
Bounding trivially (i.e. killing the oscillatory factor $X^{i t_j}$), using Cauchy--Schwarz inequality and the local Weyl law
(see \cite{Hormander:1968}, \cite[Theorem 7.2]{Iwaniec:2002}) we deduce the trivial bound $\mathfrak{S} (T,X) = O (T^2)$. This bound implies Selberg's bound \eqref{selbergbound}. The difficulty in improving \eqref{selbergbound} arises from the difficulty to improve on the trivial bound on $\mathfrak{S} (T,X)$, for any pair of points $z,w$. 

The main motivation for the pointwise improvement in \cite{cclr} came from a famous bound of H\"ormander (see \cite{Hormander:1968} for cocompact groups and \cite{Chamizo:1996} for cofinite groups), which gives that for $X=1$ and for \textit{different points} $z\neq w$ the spectral exponential sum is much smaller:
\begin{equation*} 
  \sum_{T \leq t_j < 2T}  \phi_j(z) \overline{\phi_j(w)}  = O \left(T^{1+\varepsilon} \right).
\end{equation*}
H\"ormander's bound holds in larger generality, and we refer to \cite[Section~1]{cclr} and references therein for more details on the study of growth of the spectral function. The main step for the proof of Theorem \ref{cclrmaintheorem} in \cite{cclr} was establishing the bound
\begin{equation*} 
  \sum_{t_j \leq T}  \abs{\phi_j(z)} \abs{\phi_j(w)}
  = O \left( \frac{T^2}{(\log T)^{1/4}}\right)
\end{equation*}
for the orthonormal basis $\{\phi_j\}_{j=0}^{\infty}$ consisting of Hecke--Maass cusp forms the modular group $\Gamma = \hbox{PSL}(2,\mathbb{Z})$ and for Heegner points $z =z_d, w=z_{d'}$ of different discriminants.

Our proof is heavily inspired by the above approach. Based on H\"ormander's approach, we initially tried a microlocal approach inspired by \cite{Canzani:2023} and its own antecedents. Later we tried obtaining averaged results over a neighborhood of Teichmuller space in the hopes of proving the existence of a hyperbolic structure and points $z,w$ witnessing an exponent below the $2/3$ threshold. These avenues did not produce meaningful results for us but they may still be worth exploring. Eventually, the particular quantity \eqref{hlfwvkrnl} motivated us to wonder whether strong existing half-wave kernel estimates could be useful. This avenue ultimately led to the idea of proof of Theorem \ref{ourmainthm1}.

Methods from functional and harmonic analysis have been used before to attack the hyperbolic lattice counting problem, mainly focusing in the interplay between $L^2$ and $L^p$ estimates (see for instance \cite{Chamizo:1996}). Let us denote by $\partial_R f(R)$ the partial derivative of a function $f(R)$. The first step required for our new approach is, for a fixed point $w$, to bound the $L^2$-norms of the function
\begin{equation*}
\mathfrak{B}(T, 2 \cosh R; \cdot, w)
\end{equation*}
and its partial derivative
\begin{equation*}
\partial_R \, \mathfrak{B} (T, 2 \cosh R; \cdot , w),
\end{equation*}
in the $R = \cosh^{-1}(X/2) \asymp \log X$ variable. The proof uses spatial orthogonality of eigenfunctions and thus already makes it evident that the diagonal cannot be treated using this method.

Through the use of a modified one-dimensional Sobolev inequality (Lemma \ref{BSE}) and a change of integration, we can use the aforementioned $L^2$-estimates to prove a maximal dyadic $L^2$ estimate in the $R$-variable (Theorem \ref{maximalestimate}). The use of the modified Sobolev inequality as opposed to its classical counterpart allows for the crucial cancellation of a power of of $T$ which makes the bounding terms appearing on \ref{maximalestimate} summable. Finally, a standard application of the Chebyshev inequality and the Borel-Cantelli lemma to the $z$ variable produces the large set of points $B_w$ appearing in Theorem \ref{ourmainthm1}.

\begin{remark}
The powers of the extra logarithms in Theorem \ref{ourmainthm1} come from the sum over dyadic spectral intervals and from choosing a summable threshold in the Borel-Cantelli Lemma. They are not necessarily optimal, but their improvement would require additional input. 
\end{remark}

\begin{remark}
The proof of Theorem \ref{ourmainthm1} can work for a general cofinite group $\Gamma$, working with the spectral expansion 
\begin{equation*}
 \sum_{ t_j > 0}h^{\pm}(t_j)\phi_j(z)\overline{\phi_j(w)} + \frac{1}{4\pi} \sum_{\mathfrak{a}} \int_{\mathbb{R}} h^{\pm}(t) E_{\mathfrak{a}} \left(z, \frac{1}{2} +it\right) E_{\mathfrak{a}} \overline{\left(w, \frac{1}{2} +it\right)} \, dt.
\end{equation*}
in the place of \eqref{discreteexp} (taking into account simultaneously the contribution of the discrete spectrum and the continuous spectrum in the error term $E(X;z,w)$). The Eisenstein spectrum is assembled into $L^2$ wave packets and
controlled by a generalized Plancherel formula and the full local Weyl law, combined with a cofinite dyadic maximal estimate. The upper bound of the error term depends on the points $z,w$ and is not uniform for $z$ close to a cusp. 
\end{remark}

\begin{remark} \label{remark1.5}
Maximal estimates are not new in any sense of the word. A great reference is \cite{Sogge:1993} where the versatility of maximal estimates becomes evident. The way in which we employ the Sobolev inequality is also well known, see for example Lemma $2.4.2$ in \cite{Sogge:1993}, even though similar applications have already appeared in analytic number theory, for example by Gallagher \cite[Section~9]{Iwaniec-Kowalski}. These techniques go back to the early $20^{\text{th}}$ century. However, to the best of our knowledge, they have never been employed in lattice counting problems. We believe this large circle of ideas and techniques that comes with maximal estimates will become a powerful tool in problems similar to the one studied in this paper. We have already employed it successfully in a number of forthcoming results \cite{cdks, dimakis}.
\end{remark}

\subsection{The second moment of the error term: failure of Cram\'er's law}

For the second moment of the error term, Chamizo \cite{Chamizo:1996a} proved that 
\begin{equation} \label{secondmomentconj}
  \frac{1}{X}\int_{X}^{2X} \left| E(X;z,w) \right|^2 dx \ll X  (\log X)^a
\end{equation}
with $a=2$ for any cofinite Fuchsian group. The power of $\log X$ was further reduced to $a=1$ by Cherubini \cite{Cherubini:2018} and to $a=3/4$ for pairs of Heegner points of different discriminants by Chatzakos - Cherubini - Lester - Risager \cite{cclr} for the modular group. Using numerical data for Fermat groups, Phillips and Rudnick \cite[Section 3.8]{PhillipsRudnick:1994} predicted that \eqref{secondmomentconj} should hold for $z=w$ with an upper bound $\ll X$. If true, this would be an analogue of Cram\'er's result \cite{cramer} for the Euclidean Gauss circle problem. 

In the following two results, we study the second moment of the error term for cocompact groups. First, we prove a stronger upper bound for the second moment of the error term for almost all $z,w$, Second, and maybe rather unexpectedly, we prove that, contrary to Rudnick and Phillips' prediction, Cram\'er's result does not hold for the error term on the diagonal $z=w$ in the case of cocompact groups $\Gamma$. 

\begin{theorem} \label{cramercloser}
Let $\Gamma$ be any cocompact group and fix a point $w \in \mathbb{H}^2$. Then, there exists a set $B_w \subset \Gamma \backslash \mathbb{H}^2$ of full hyperbolic measure, such that for every $z \in B_w$ the bound
\begin{equation*} 
  \frac{1}{X}\int_{X}^{2X} \left| E(X;z,w) \right|^2 dx \ll X   \log \log X
\end{equation*}
holds as $X\to \infty$, uniformly for all sufficiently large $X$.
\end{theorem} 

\begin{theorem} \label{cramerfailure}
Let $\Gamma$ be a cocompact group. Then, for every $z \in \mathbb{H}^2$ we have
\begin{eqnarray} \label{diagonalobstruction}
\limsup\limits_{X \to +\infty} \ \frac{1}{X^2}\int_{X}^{2X} \left| E(X;z,z) \right|^2 dx = +\infty.
\end{eqnarray}
\end{theorem}

\begin{remark}
The proof of Theorem \ref{cramerfailure} does not specify a rate of divergence for the left side of \eqref{diagonalobstruction}. 
It should be mentioned that it's easier to deduce a higher dimensional analogue of Theorem \ref{cramerfailure} due to the abundance of eigenvalue, see \cite[Corollary~1.1]{katsivelos}. The question of the asymptotic behavior of the second moment remains open in all dimensions. 
\end{remark}

\subsection{Higher dimensions}

For $n \geq 3$ let $\mathbb{H}^n$ denote the classical $n$-dimensional real hyperbolic space. Let also $\Gamma \subset {{\mathrm{SO}^+(1,n)} }$ be a discrete cocompact group of isometries (lattice) acting on $\mathbb{H}^n$ and denote by $\rho(z,w)$ the hyperbolic distance of two points $z,w \in \mathbb{H}^n$. Similarly to the $2$-dimensional case, the counting function $N(X;z,w) = \# \{ \gamma \in \Gamma : 2 \cosh \rho( z, \gamma w) \leq X \}$ in higher dimensions satisfies an asymptotic behavior 
\begin{equation*}
N (X;z,w) = M (X; z,w)  + E (X;z,w),
\end{equation*}
where the main term $M (X; z,w)$ is given by a finite sum over the small eigenvalues 
$\lambda_j \leq (n-1)^2/4 $
and the error term $E (X;z,w)$ (containing the contribution of the large eigenvalues and the continuous spectrum) satisfies the bound
\begin{eqnarray*} 
 E (X;z,w) = O_{\Gamma} \left(X^{n-2 + \frac{2}{n+1}} (\log X)^{\frac{3}{n+1}}\right),
\end{eqnarray*}
see for instance \cite{hillparnovski, laxphillips}. This can be understood as the higher dimensional analogue of Selberg's $X^{2/3}$-bound. One may feel tempted to expect square root cancellation in higher dimensions:
\begin{equation} \label{conjecturen}
 E (X;z,w) = O_{\epsilon} \left(X^{\frac{n-1}{2}+\epsilon} \right)
\end{equation}
for any $z$ and $w$; however Phillips and Rudnick proved that this does not hold globally for $n \geq 4$. In \cite{PhillipsRudnick:1994} they constructed an explicit arithmetic lattice $\Gamma \subset \hbox{SO}^+(1,n) \cap {\hbox{SL}_{n+1}( {\mathbb Z})}$ satisfying
\begin{eqnarray} \label{lowerboundphillipsrudnick}
 E_{\Gamma}(X;z_0,z_0) = \Omega(X^{n-2})
\end{eqnarray}
for a special arithmetic point $z_0$. They emphasize that this kind of jump in the error term is not spectral, instead they deduce this lower bound by clearly arithmetic means (number of solutions of quadratic forms). However, for general lattices $\Gamma$ and points $z,w$ they only proved much weaker lower bounds, consistent with \eqref{conjecturen}.

In our final result, we prove that square root cancellation \eqref{conjecturen} holds for cocompact groups in higher dimensions for almost all pairs of points $z,w$.

\begin{theorem} \label{higherdim}
Let $\Gamma$ be a cocompact lattice and fix a point $w \in \mathbb{H}^n$. Then, there exists a set $B_w \subset \Gamma \backslash \mathbb{H}^n$ of full hyperbolic measure, such that for every $z \in B_w$ and for every $\epsilon >0 $ the bound
\begin{eqnarray*}
E(X;z,w) = O_{\Gamma, w,z,\epsilon}  \left(X^{\frac{n-1}{2}} (\log X)^{3/2} (\log \log X)^{1/2+\epsilon}\right)
\end{eqnarray*}
holds as $X \to \infty$, uniformly for all sufficiently large $X$. In particular, the bound \eqref{conjecturen} holds for almost all pairs of points. 
\end{theorem}

\begin{remark}
As in $2$ dimensions, Theorem \ref{higherdim} holds for a general cofinite lattice $\Gamma$.
\end{remark}

\begin{remark}
Theorem \ref{higherdim} and lower bound \eqref{lowerboundphillipsrudnick} indicate a crucial dichotomy between the diagonal $z=w$ and the non-diagonal $z\neq w$ case in higher dimensions, where the behavior of the error term remains more mysterious. We refer to \cite{cherubinikatsivelos, hillparnovski, katsivelos, laxphillips, PhillipsRudnick:1994} and references therein for more detailed discussions on the higher dimensional problem. 
\end{remark}

\subsection{Limitations of our method}

The argument of our proof guarantees that the conjectural upper bound \eqref{mainconjecture} holds for almost all pairs of points $(z,w)$. Nevertheless, it cannot guarantee the bound for any \textit{specific} preassigned pair of points. The set $B_w$ in Theorem \ref{ourmainthm1} does in general depend on $w$, and as such, does not produce any explicit point $z \in B_w$, a single full-measure set valid for all $w$ or any information for the diagonal case $z=w$. The last case remains as the most difficult one, and new ideas are necessary in order to treat the general case. The same conclusion holds for Theorems \ref{cramercloser} and \ref{higherdim}.

\begin{remark}
For a complete study of the hyperbolic lattice counting problems, we refer also to \cite{biro1, Biro:2024, Chatzakos:2017, Cherubini:2018, CherubiniRisager:2018, cherubinikatsivelos, gunther, katsivelos, laxphillips, PetridisRisager:2017, PetridisRisager:2018a} for various important results and modifications for lattice counting problems on classical hyperbolic spaces of rank one, and to \cite{BlomerLutsko:2024, duke, eskinmcmullen, gorodnik} for lattice counting problems in homogeneous spaces of higher rank. 
\end{remark}

\subsection{AI Declaration} The main idea in the proof of the pointwise imrpovement (Theorem \ref{ourmainthm1}) is the use of the Sobolev inequality to obtain a maximal $L^2$ estimate in the radius variable from which the result follows by a standard application of the Borel-Cantelli Lemma. This idea was obtained after extensive interactions with ChatGPT-5.6 Sol. The proof of \ref{cramerfailure} was obtained exclusively by ChatGPT-5.6 Sol. The authors formulated the project, evaluated and revised the model’s suggestions, independently verified all mathematical arguments, and wrote the final manuscript. The authors take full responsibility for the correctness of all statements and proofs in this paper.

\subsection{Acknowledgments} We would like to thank Yiannis Petridis for useful conversations and for constant support for the duration of this project. We also thank Giacomo Cherubini for useful comments and Morten Risager for his support. 
The first author is supported by the Hellenic Foundation for Research and Innovation (H.F.R.I.) under the “3rd Call for H.F.R.I. Research Projects to support Faculty Members \& Researchers” (Project Number: 25622). The second author would like to thank Aristomenis Papadopoulos for helping him formally verifying parts of the proofs of the paper and for many fruitful conversations.

\section{Spectral theoretic background} \label{section2}

In this section we set the necessary spectral background. The material included in this section is standard, but we include a detailed presentation for reasons of completeness. 

\subsection{The automorphic kernel and the spherical transform} 

We denote by $-\Delta_{\Gamma}$ the positive automorphic Laplacian of the Riemann surface $\Gamma \backslash \mathbb{H}^2$ and by $\{ \phi_j \}_{j=0}^{\infty}$ an orthonormal basis for the discrete spectrum of $-\Delta_{\Gamma}$ with eigenvalues $ \{\lambda_j \}_{j=0}^{\infty}$ in increasing order $\lambda_0 = 0 < \lambda_1 \leq \lambda_2 \leq ...$ and $\lambda_j \to \infty$.
We also use the standard notation  $\lambda_j = s_j (1-s_j) =  1/4 +t_j^2$. The small eigenvalues $\lambda_j <1/4$ correspond to parameters $ 0 \leq s_j \leq 1$ and $|\Im(t_j)| \leq 1/2$. Large eigenvalues $\lambda_j \geq 1/4$ correspond to $\Re(s_j)=1/2$, equivalently to $t_j \geq 0$. Recall also the change of variable $X = 2 \cosh R \sim e^R$.

Without loss of generality assume $R\geq 1$. Let $k$ be the characteristic kernel depending on the hyperbolic distance between two points, given explicitly by
\begin{eqnarray} \label{kernel}
 k_R (u(z, w)) = k_R (u) = \left\{ \begin{array}{rcl}
1, & \mbox{for} &  u \leq (\cosh R -1)/2,
\\ 0, & \mbox{for} & u > (\cosh R -1)/2,
\end{array} \right. 
\end{eqnarray}
where $u(z,w)$ is given by \eqref{pointpairinvariant}. The automorphic kernel 
\begin{equation*}
K_R (z,w) =  K(z,w)=\sum_{\g\in \G}k(u(z,\gamma w)) 
\end{equation*}
satisfies $K(z,w) = N(X;z, w) $. Applying Selberg's pre-trace formula \cite[Theorem 7.4]{Iwaniec:2002} to $K(z,w)$ we see that the spherical transform $h_R(t)$ of $k(u) = k_R (u) $ does not decay quickly enough for this to be applied directly, because the decay of $h_R(t)$ in the $t$-variable is not strong enough to analyze effectively the right-hand side of the pre-trace formula. However, we can study the counting function approximating smoothly the automorphic kernel $K(z,w)$. 

The spherical (or Selberg/Harish-Chandra) transform $h_R(t)$ of the characteristic function is given in terms of the associated Legendre function of the first kind 
\begin{equation} \label{eq:hRdef}
  h_R(t)= 2\pi\sinh(R)P_{-\frac{1}{2}+it}^{-1}(\cosh R),
\end{equation}
(see \cite[Equation (2.7)]{Chamizo:1996a}). The following lemma summarizes some basic properties of $h_R(t)$ and $\partial_R \, h_R(t)$ (the partial derivative in the $R$-variable), for $R \to \infty$ and for $R$ very small. 
 \begin{lemma} \label{sphericalproperties}
The spherical transform $h_R(t)$ is entire in $t \in \mathbb{C}$ and satisfies
\begin{equation*}
h_R(t) =  O \left((1+R) \, e^{R/2+R\abs{\Im(t)}}\right). 
\end{equation*}
Moreover:
\\
(a) for real $|t| \geq 1, R \geq 1$ we have
\begin{equation*} 
  h_R(t)=O\left(\frac{e^{R/2}}{(1+{\abs{t})}^{3/2}}\right), \ \ \ \ 
  \partial_R \, h_R(t) = \left( \frac{e^{R/2}}{(1+{\abs{t})}^{1/2}}\right),
\end{equation*}
(b) for real $t \neq 0$ and small $\delta>0$ we have 
\begin{equation*} 
m_{\delta}(t):= \frac{h_{\delta}(t)}{4\pi\sinh^2(\delta/2)} = O\left(\frac{1}{(1+{\abs{\delta t})}^{3/2}}\right). 
\end{equation*}
\end{lemma}
\begin{proof}
The proof follows immediately from \cite[Lemma~2.4]{Chamizo:1996a}. The only non-trivial estimate is the second estimate of part (a), which can be deduced from the explicit equation of the spherical transform (see \cite[eq.~(2.6)]{Chamizo:1996a})
\begin{equation*}
h_R(t) =  4 \sqrt{2} \int_{0}^{R} \sqrt{\cosh R - \cosh r} \cos(tr) \, dr.
\end{equation*}
Differentiating the integrable endpoint singularity gives
\begin{equation} \label{derivativeintegral}
\partial_R h_R(t) =  2 \sqrt{2} \sinh R \int_{0}^{R} \frac{\cos(tr)}{\sqrt{\cosh R - \cosh r}}  \, dr.
\end{equation}
Since $h_R(t)$ is even, we can assume that $t \geq 1$. For $v \in (0,1]$ and $R \geq 1$ we have
\begin{equation*}
\cosh R - \cosh (R-v) \asymp v \sinh R.
\end{equation*}
We conclude that for $r \in [R-1/t, R]$ the contribution of the integral in \eqref{derivativeintegral} is bounded by $O(e^{-R/2} t^{-1/2})$, whereas for $r \in [0, R-1/t]$ the function $f(r) = (\cosh R - \cosh r)^{-1/2}$ is increasing (notice that for $R=t=1$ the interval is empty). Using integration by parts we bound the contribution of the integral by 
\begin{equation*}
t^{-1} \left( f \left(R- \frac{1}{t}\right) + \int_{0}^{R-1/t} |f'(r)| dr\right) \leq 2 t^{-1} f \left(R- \frac{1}{t} \right)  \ll e^{-R/2} t^{-1/2}.
\end{equation*}
Multiplying by $ \sinh R$ completes the proof.
\end{proof}

\subsection{Smoothing} 

As we have already mentioned, we cannot apply the pre-trace formula directly to the automorphic kernel $K_R(z,w)$. For that reason, we define smooth approximations $K_{\pm}(z,w)$ of the kernel $K$ by smoothing out the kernel $k$. This gives better convergence on the spectral side, and allows us to apply the pre-trace formula for $K_{\pm}(z,w)$.

For any two kernels $k_1$ and $k_2$ let $k_3 = k_1 * k_2$ define their hyperbolic convolution, given by
\begin{equation} \label{convolution}
  (k_1* k_2) (u(z,w))=\int_{\mathbb{H}^2} k_1(u(z,v))k_2(u(v,w))d\mu(v).
\end{equation}
The spherical transform is multiplicative in the following sense: if we denote by $h_i(t)$ the spherical transform of $k_i$, then the spherical transform of \eqref{convolution} is the product of the corresponding spherical transforms:
\begin{equation}
  \label{eq:convolution}
  h_{3}(t)=h_{1}(t) \, h_{2}(t).
\end{equation}
(see \cite{Chamizo:1996a}).
Assume $0<\delta <1$ is a sufficiently small parameter. Similarly to the definition of $k_R$, we consider the (appropriately normalized) kernel
\begin{equation*}
  k_\delta(u)= \frac{1}{4\pi \sinh^2(\delta/2)} \mathbf{1}_{[0,(\cosh \delta-1)/2]}(u),
\end{equation*}
where by $\mathbf{1}_{A}$ we denote the characteristic function of $A$. This normalized kernel by definition satisfies
\begin{equation*}
\int_{\mathbb{H}^2} k_\delta(u(z,w))d\mu(z)=1.
\end{equation*}
Now let $X \gg 1$ and recall $R = \arccosh (X/2) >0$. We define two smooth approximations of the characteristic kernel
\begin{equation} \label{approximations}
  k^{\pm}(u)= \left(\mathbf{1}_{[0,(\cosh(R \pm \delta )-1)/2]}*k_\delta \right)(u).
\end{equation}
Using the triangle inequality for the metric $\rho$ we see (see also \cite[eq.~(5.4)]{PetridisRisager:2017}) that
\begin{equation}\label{ineqkpm}
  k^-(u(z,w))\leq k(u(z,w))\leq k^+(u(z,w)).
\end{equation}
By relations \eqref{eq:convolution} and \eqref{approximations} we find that the Selberg--Harish-Chandra transforms of $k^{\pm}$ are given by
\begin{equation*}
  h_{R,\delta}^{\pm} (t) :=\frac{h_{R \pm \delta}(t) \, h_\delta (t)}{4\pi\sinh^2(\delta/2)} = h_{R \pm \delta}(t) \, m_{\delta}(t).
\end{equation*}
In Lemma \ref{sphericalproperties} we have already analyzed the behavior of $h_R$ for $R$ large and of $m_{\delta}$ for $\delta$ small (see also \cite{PetridisRisager:2017}). Using part (a) of Lemma \ref{sphericalproperties} we find that, in the strip $\abs{\Im(t)}\leq 1-\e$, $h^{\pm}$ satisfies
\begin{equation*} 
h_{R,\delta}^{\pm}   (t)=O_{X,\delta}\left(\frac{1}{(1+\abs{t})^3}\right).
\end{equation*}
This justifies that we can use the pre-trace formula. Finally, working as in \cite[Equations (6.2)-(6.4)]{PetridisRisager:2018a} we derive the formulas
\begin{eqnarray}   \label{good-bounds-for-hpm}
 h_{R,\delta}^{\pm}     \left(\tfrac{i}{2}\right) &=& \pi X+O(1+\delta X) = \pi e^R + O \left(1+\delta e^{R}\right), \nonumber \\  
 h_{R,\delta}^{\pm} (t) &=& O(X^{1/2}\min(\abs{t}^{-3/2}, \delta^{-3/2}\abs{t}^{-3}, \log X)) \\
    &=& O(e^{R/2}\min(\abs{t}^{-3/2}, \delta^{-3/2}\abs{t}^{-3}, R)) . \nonumber
\end{eqnarray}

\subsection{Applying the pre-trace formula}

From \eqref{ineqkpm} we deduce that the automorphic kernels
\begin{equation*} 
  K_X^{\pm}(z,w):=\sum_{\g\in \G}k^{\pm}(u(z,\g w))
\end{equation*}
satisfy
\begin{equation*}
  K_X^-(z,w)\leq N(X;z,w)\leq K_X^+(z,w).
\end{equation*}
We can now apply Selberg's pre-trace formula  \cite[Theorem~7.4]{Iwaniec:2002} to the smooth automorphic kernels $K_X^{\pm}(z,w)$. Let
\begin{eqnarray*} 
M (X;z,w)&=&\sum_{\frac{1}{2} < s_j \leq 1} \sqrt{\pi} \frac{\Gamma \left(s_j-1/2\right)}{\Gamma(s_j+1)} \phi_j(z) \overline{\phi_j(w)} X^{s_j} \nonumber \\
&&+ \sum_{\lambda_j=1/4} h_R(0) \phi_j(z) \overline{\phi_j(w)}.
\end{eqnarray*}
This is the well understood main term of $N(X;z,w)$, which appears as the contribution of the small eigenvalues $\lambda_j \in [0,1/4]$ in the spectral expansion of $K_X$ (notice that in our notation we have included the contribution of the eigenvalue $\lambda_j=1/4$ in the main term). Selberg's pre-trace formula allows us to expand 
  \begin{equation} \label{kpmexpansion}
    K^\pm(X;z,w)= \sum_{\lambda_j}h^{\pm}(t_j)\phi_j(z)\overline{\phi_j(w)}.
  \end{equation}
Let us set the contribution of small eigenvalues in \eqref{kpmexpansion} as
\begin{equation*}
M^{\pm} (X;z,w) = \sum_{0 \leq \lambda_j  \leq 1/4}h^{\pm}(t_j)\phi_j(z)\overline{\phi_j(w)}.
\end{equation*}
A standard estimate (see \cite{Chamizo:1996a}) gives
\begin{equation*}
    M (X;z,w)-M ^{\pm}(X;z,w)=O\left(X \delta + X^{1/2} \log X\right).
\end{equation*}
The contribution of the eigenvalue $\lambda_j = 1/4$ (if present) is bounded using Lemma \ref{sphericalproperties}, since $h(0), h^{\pm} (0)  = O(R \, e^{R/2}) = O(X^{1/2} \log X)$. 
After defining the smooth approximations of the error term to be
\begin{equation*}
    E^{\pm} (X;z,w) = K^{\pm} (X;z,w) - M^{\pm}(X;z,w),
\end{equation*}
we derive 
\begin{equation} \label{crucialrelation}
    E (X;z,w) \ll E^{\pm}(X;z,w)+O\left(X\delta + X^{1/2} \log X\right).
\end{equation}
For the rest of the paper we focus on the contribution of the large eigenvalues in the smooth error terms $E^{\pm}(X;z,w)$, since these eigenvalues play the crucial role for our purposes. 

\subsection{The smooth error term and spectral exponential sums}

We split the contribution of the discrete spectrum in appropriate ranges, separately for $\abs{t_j} \leq 1$ and for $ 1< \abs{t_j}$; we split the later dyadically. After expanding, we need to bound the series
\begin{eqnarray} \label{discreteexp}
E^{\pm} (X;z,w) &=& \sum_{ t_j > 0}h^{\pm}(t_j)\phi_j(z)\overline{\phi_j(w)}  \nonumber \\ 
&=& \sum_{ t_j \leq 1}h^{\pm}(t_j)\phi_j(z)\overline{\phi_j(w)} +  \sum_{k=0}^{\infty} B^{\pm} (2^k, R;z,w), 
\end{eqnarray}
where we simplify notation setting 
\begin{equation} \label{Bsums}
B^{\pm} (T,R;z,w) = \mathfrak{B}^{\pm}  (T, 2 \cosh R;z,w) = \sum_{T \leq t_j < 2T} h_{R, \delta}^{\pm} (t_j) \phi_j(z)\overline{\phi_j(w)},
\end{equation}
defined similarly as in \eqref{BTXdef}. Notice that the contribution of $\abs{t_j} \leq 1$ is trivially bounded using Lemma \ref{sphericalproperties} by $O_{\Gamma,z,w} (X^{1/2} \log X)$, since it is a finite sum. From equations \eqref{crucialrelation} and \eqref{discreteexp} we conclude
\begin{equation*} 
    E (X;z,w) \ll \sum_{k=0}^{\infty} B^{\pm} (2^k, R;z,w)   +O\left(X\delta + X^{1/2} \log X\right).
\end{equation*}
Estimating the first sum will occupy the rest of our proof.

\section{The maximal estimate and the proof of Theorem \ref{ourmainthm1}}

Let us fix some extra notation. Recall that $R\geq 1$. We split the interval 
\begin{equation*}
[1, +\infty) = \bigcup_{n=1}^{\infty} I_n, \ \ \text{where} \ \ I_n= [n, n+1].
\end{equation*}
Instead of choosing $\delta$ smmothly on $R$, we can make the following simpler choice: for $R\in [n,n+1)$, let 
\begin{equation*}
\delta(R) = \delta_n = e^{-\frac{n+1}{2}}.
\end{equation*} 
Here increasing $R$ from $[n,n+1)$ to $[n+1,n+2)$ makes $\delta(R)$ jump; however this discontinuity does not affect the argument, but rather makes our proof simpler. The main result of this section is the following.

\begin{theorem}   \label{maximalestimate}
Fix $w\in \mathbb{H}^2$ and $\eta>0$. There exists a Borel set $B_{w,\eta}\subset \Gamma \backslash \mathbb{H}^2$ of full hyperbolic measure such that for all $z\in B_{w,\eta}$ the following estimate holds: 
\begin{equation} \label{maximalupperbound}
\left|\sum\limits_{t_j\geq 1} h^{\pm}_{R, \delta(R)}(t_j)\phi_j(z)\overline{\phi_j(w)}\right| \ll_{\Gamma,w,z,\eta} e^{\frac{R}{2}} \, R^{\frac{3}{2}} \, (\log(R))^{\frac{1}{2}+\eta},
\end{equation}
for \textbf{every} sufficiently large real $R$. 
\end{theorem}

The proof of Theorem \ref{maximalestimate} occupies the rest of this section. Clearly, the upper bound \eqref{maximalupperbound}
is equal to $X^{1/2} (\log X)^{3/2} (\log \log X)^{1/2+\eta}$.

\subsection{$L^2$-bounds for dyadic spectral sums}
In this subsection we prove $L^2$-bounds for the dyadic spectral sums $B_n(T,R)$ and $\partial_R B_n(T,R)$ (defined in \eqref{Bnpm}). For $R\in I_n$ define the normalized spherical transform
\begin{equation} \label{normalizedspherical}
c^{\pm}_{n,t}(R) := e^{-\frac{R}{2}}h^{\pm}_{R,\delta_n}(t) 
\end{equation}
and the weight
\begin{equation*} 
  Q_{n,T} := (1+\delta_n T)^{-\frac{3}{2}}.   
\end{equation*}
\begin{lemma}\label{cpmdcpm}
The following estimates hold uniformly for $t\geq 1$ and for $R \in I_n$:
\begin{equation*}
\begin{split}
|c_{n,t}^{\pm}(R)| &\ll t^{-\frac{3}{2}}Q_{n,t},\\\\
|\partial_R c^{\pm}_{n,t}(R)| &\ll t^{-\frac{1}{2}}Q_{n,t}. 
\end{split}
\end{equation*}
\end{lemma}

\begin{proof}
The first inequality follows directly from parts (a) and (b) of Lemma \ref{sphericalproperties}.  In order to obtain the derivative estimate, we use \eqref{normalizedspherical} to write
\begin{equation*}
\partial_R c^{\pm}_{n,t}(R) = -\frac{1}{2}c_{n,t}^{\pm}(R) + e^{-\frac{R}{2}}(\partial_Hh_H)(t)|_{H = R\pm \delta_n}  m_{\delta_n}(t).
\end{equation*}
The full estimate follows from the first estimate and the second bound of part (a) of Lemma \ref{sphericalproperties}, combined with bounds \eqref{good-bounds-for-hpm}.
\end{proof}
Now, for $k\geq 0$ and  $T=2^k$ define the following convoluted dyadic sums, which are modifications of the dyadic sums \eqref{Bsums}:
\begin{equation} \label{Bnpm}
B_n^{\pm}(T,R; z,w) := \sum\limits_{T\le t_j<2T} c_{n,t_j}^{\pm} (R) \,  \phi_j(z) \, \overline {\phi_j(w)}. 
\end{equation}
We have the following $L^2$-bounds for these convoluted dyadic sums (considered as functions of $z$: $B_n^{\pm}(T,R;w) := B_n^{\pm}(T,R; \cdot,w)$).
\begin{proposition}\label{L2bounds}
For every $n\geq 1$, every dyadic $T\geq 1$ and every $R\in I_n$ the following bounds hold: 
\begin{equation*}
\begin{split}
\|B_n^{\pm}(T,R; w)\|_2^2:=& \|B_n^{\pm}(T,R; w)\|_{L^2(\Gamma \backslash \mathbb{H}^2)}^2 \ll_{\Gamma} T^{-1}Q_{n,T}^2, \\\\
\|\partial_RB_n^{\pm}(T,R; w)\|_2^2:=& \|\partial_RB_n^{\pm}(T,R; w)\|_{L^2(\Gamma \backslash \mathbb{H}^2)}^2 \ll_{\Gamma} TQ_{n,T}^2.
\end{split}
\end{equation*}
\begin{proof}
For the first bound we compute the $L^2$-norm by
\begin{eqnarray*}
\|B_n^{\pm}(T,R; w)\|_{L^2(\Gamma \backslash \mathbb{H}^2)}^2 &=& \int_{\Gamma \backslash \mathbb{H}^2} \sum\limits_{T\le t_j<2T} \sum\limits_{T\le t_l<2T} c_{n,t_j}^{\pm}(R)c_{n,t_l}^{\pm}(R) \phi_j(z)\overline{\phi_l(z)}\overline{\phi_j(w)}\phi_l(w) \,d\mu(z)\\
&=& \sum\limits_{T\le t_j<2T}|c_{n,t_j}^{\pm}(R)|^2|\phi_j(w)|^2,
\end{eqnarray*}
using the fact that $\{\phi_j(z)\}_j$ is an orthonormal basis. From Lemma \ref{cpmdcpm} we get the inequality 
\begin{equation*}
|c_{n,t_j}^{\pm}(R)|^2 \ll T^{-3}Q_{n,T}^2,
\end{equation*}
and therefore 
\begin{eqnarray*}
\|B_n^{\pm}(T,R; w)\|_{L^2(\Gamma \backslash \mathbb{H}^2)}^2 &\ll_{\Gamma}& T^{-3}Q_{n,T}^2 \sum\limits_{T\le t_j<2T} |\phi_j(w)|^2 \\
&\ll_{\Gamma, w}& T^{-3}Q_{n,T}^2T^2 \\
&=& T^{-1}Q_{n,T}^2,
\end{eqnarray*}
where the last inequality is obtained by applying the local Weyl law on the interval $[T,2T)$. An identical orthogonality computation gives us 
\begin{equation*}
\|\partial_RB_n^{\pm}(T,R; w)\|_{L^2(\Gamma \backslash \mathbb{H}^2)}^2 = \sum\limits_{T\le t_j<2T}|\partial _Rc_{n,t_j}^{\pm}(R)|^2|\phi_j(w)|^2.
\end{equation*}
Using the second inequality from Lemma \ref{cpmdcpm} and the local Weyl law gives the required bound. 
\end{proof}
\end{proposition}

\subsection{The unit interval Sobolev estimate}

A key part of the proof is the following Sobolev-type estimate. As mentioned in Remark \ref{remark1.5}, such estimates have already been employed in analytic number theory.  
\begin{lemma}\label{BSE}
Let $I\subset R$ be a closed interval of unit length and $F \in H^1(I,\C):= L^{2,1}(I,\C)$. Then 
\begin{equation}\label{BSEineq}
\sup\limits_{r\in I} |F(r)|^2 \le \int_I|F(u)|^2\,du + 2\left(\int_I|F(u)|^2\,du\right)^{\frac{1}{2}}\left(\int_I |F'(u)|^2\,du\right)^{\frac{1}{2}}.
\end{equation}
\end{lemma}
\begin{proof}
There exists $r_0\in I$ such that 
\begin{equation*}
|F(r_0)|^2 \le \int _I |F(u)|^2\,du. 
\end{equation*}
Since $F$ is absolutely continuous, for almost any $u\in I$ we get
\begin{equation*}
\frac{\,d}{\,du} |F(u)|^2 = 2\Re (F'(u)\overline{F(u)}). 
\end{equation*}
Integrating, applying the triangle inequality and then the Cauchy-Schwarz inequality we get 
\begin{equation*}
\begin{split}
|F(r)|^2 &\le |F(r_0)|^2 + 2\Bigg| \int\limits_{r_0}^r\Re(F'(u)\overline{F(u)})\,du\Bigg| \\
& \le \int_I |F(u)|^2\,du + 2\int |F'(u)||F(u)|\,du \\
&\le \int_I |F(u)|^2\,du + 2\|F(u)\|_2\|F'(u)\|_2.
\end{split}
\end{equation*}
which is the desired inequality \eqref{BSEineq}.
\end{proof}
\begin{remark}
We want to emphasize the fact that the more familiar Sobolev inequality 
\begin{equation*}
\|F\|_{\infty}^2 \ll \|F\|^2+ \|F'\|^2
\end{equation*}
is not good enough for our purposes and it is precisely the appearance of the product of the two norms on the right hand side instead of the sum of squares of the norms which allows for the miraculous cancellation $T^{-1/2}\cdot T^{1/2} =1$ described below. 
\end{remark}
\begin{proposition}[The maximal estimate]\label{mainineq}
Uniformly in $n$, $T$, $w\in \Gamma \backslash \mathbb{H}^2$ and the choice of sign, the following maximal estimate holds:
\begin{equation*}
\int_{\Gamma \backslash \mathbb{H}^2} \sup\limits_{R\in I_n}|B_n^{\pm}(T,R;z,w)|^2\,d\mu(z) \ll_{\Gamma} Q_{n,T}^2. 
\end{equation*}
\end{proposition}
We emphasize the importance of this estimate, as it is the key result that gives the uniform pointwise upper bound for the error term of the lattice counting problem.
\begin{proof} Let us define the non-negative functions 
\begin{equation*}
a(z) = \int_{I_n} |B_n^{\pm}(T,R;z,w)|^2\,dR
\end{equation*}
and 
\begin{equation*}
b(z) = \int_{I_n} |\partial_R B_n^{\pm}(T,R;z,w)|^2\,dR. 
\end{equation*}
The Sobolev inequality from Lemma \ref{BSE} and Cauchy-Schwarz give 
\begin{equation*}
\begin{split}
\int_{\Gamma \backslash \mathbb{H}^2} \sup\limits_{R\in I_n}|B_n^{\pm}(T,R;z,w)|^2\,d\mu(z) &\le \int_{\Gamma \backslash \mathbb{H}^2} a(z)\,d \mu(z) + 2\int_{\Gamma \backslash \mathbb{H}^2} a(z)^{\frac{1}{2}}b(z)^{\frac{1}{2}}\,d \mu(z) \\
&\le \int_{\Gamma \backslash \mathbb{H}^2} a(z)\,d \mu(z) \\ &+ 2\left(\int_{\Gamma \backslash \mathbb{H}^2} a(z)\,d \mu(z)\right)^{\frac{1}{2}}\left(\int_{\Gamma \backslash \mathbb{H}^2} b(z)\,d \mu(z)\right)^{\frac{1}{2}}. 
\end{split}
\end{equation*}
Now Tonelli's theorem allows us to interchange the order of integration so that 
\begin{equation*}
\begin{split}
&~~~~\int_{\Gamma \backslash \mathbb{H}^2} a(z)\,d \mu(z) + 2\left(\int_{\Gamma \backslash \mathbb{H}^2} a(z)\,d \mu(z)\right)^{\frac{1}{2}}\left(\int_{\Gamma \backslash \mathbb{H}^2} b(z)\,d \mu(z)\right)^{\frac{1}{2}} \\&=\int_{I_n} \|B_n^{\pm}(T,R;z,w)\|_2^2\,dR \\
&+ 2\left(\int_{I_n} \|B_n^{\pm}(T,R;w)\|^2_2\,dR \right)^{\frac{1}{2}}\left(\int_{I_n} \|\partial_RB_n^{\pm}(T,R;w)\|_2^2\right)^{\frac{1}{2}}\\
&\ll_{\Gamma}  (T^{-1} + 2T^{-1/2}\cdot T^{1/2}) \, Q_{n,T}^2 \ll_{\Gamma} Q_{n,T}^2,
\end{split}
\end{equation*}
where in the last step we used Proposition \ref{L2bounds}. 
\end{proof}

\subsection{Summation over the dyadic spectral intervals} 

Define 
\begin{equation*}
S_n^{\pm}(R;z,w) := \sum \limits_{k=0}^{\infty} B_n^{\pm}(2^k,R;z,w)
\end{equation*}
and
\begin{equation*}
G_n^{\pm}(z,w) = \sup\limits_{R\in I_n} |S_n^{\pm}(R;z,w)|.
\end{equation*}
For fixed $n$ we can use Lemma \ref{cpmdcpm}, Cauchy-Schwarz and the local Weyl law to bound the absolute contribution of each dyadic sum 
\begin{equation*}
\begin{split}
|B_n^{\pm}(T,R;z,w)| &= \left|\sum\limits_{T\le t_j<2T} c_{n,t_j}^{\pm}(R)\phi_j(z)\overline {\phi_j(w)} \right|\\
&\ll_{\Gamma} T^{-\frac{3}{2}}Q_{n,T} \sum\limits_{T\le t_j<2T} |\phi_j(z)||\overline {\phi_j(w)}|\\
&\ll_{\Gamma} T^{-\frac{3}{2}}Q_{n,T} \left(\sum\limits_{T\le t_j<2T} |\phi_j(z)|^2\right)^{\frac{1}{2}}\left(\sum\limits_{T\le t_j<2T} |\phi_j(w)|^2\right)^{\frac{1}{2}}\\
&\ll_{\Gamma} T^{-\frac{3}{2}}Q_{n,T} T^2 = T^{\frac{1}{2}}Q_{n,T}. 
\end{split}
\end{equation*}
The series $T^{\frac{1}{2}}Q_{n,T}$ is summable over dyadic $T= 2^k$. Therefore the series $G_n^{\pm}(z,w)$ converges absolutely and uniformly on $I_n\times \Gamma\backslash\mathbb{H}^2 \times \Gamma \backslash \mathbb{H}^2$ and its sum is continuous in $(z,w)$ and therefore Borel measurable. 
\begin{proposition}\label{nplus1}
For each fixed $w\in \Gamma \backslash \mathbb{H}^2$,
\begin{equation*}
\int_{\Gamma \backslash \mathbb{H}^2} |G_n^{\pm}(z,w)|^2\,d\mu(z) \ll_{\Gamma} (1+n)^2. 
\end{equation*}
\end{proposition}
\begin{proof} Write the finite block sums 
\begin{equation*}
S_{n,K}^{\pm}(R;z,w) := \sum \limits_{k=0}^{K} B_n^{\pm}(2^k,R;z,w),
\end{equation*}
and define
\begin{equation*}G_{n,K}^{\pm}(z,w) = \sup\limits_{R\in I_n} |S_{n,K}^{\pm}(R;z,w)|.
\end{equation*}
Then, applying the Minkowski inequality and the maximal estimate of Proposition \ref{mainineq} we obtain 
\begin{equation*}
\begin{split}
\left(\int_{\Gamma \backslash \mathbb{H}^2} |G_{n,K}^{\pm}(z,w)|^2\,d\mu(z)\right)^{\frac{1}{2}} &= \left(\int_{\Gamma \backslash \mathbb{H}^2} \sup\limits_{R\in I_n} \left|\sum \limits_{k=0}^{K} B_n^{\pm}(2^k,R;z,w)\right|^2\,d\mu(z) \right)^{\frac{1}{2}}\\
&\ll_{\Gamma} \sum\limits_{k=0}^K \left(\int_{\Gamma \backslash \mathbb{H}^2}\sup\limits_{R\in I_n}\left|B_n^{\pm}(2^k,R;z,w)\right|^2\,d\mu(z)\right)^{\frac{1}{2}}\\
&\ll_{\Gamma} \sum\limits_{k=0}^K Q_{n,2^k} \\
&\ll_{\Gamma} \sum\limits_{k=0}^{\infty} Q_{n,2^k} \ll_{\Gamma} (1+n).
\end{split}
\end{equation*}
The last inequality follows if we split the series for $k\geq \lfloor n/2\log2\rfloor$ and $k\leq\lfloor n/2\log2\rfloor$ and treat the two cases separately. Finally, by Fatou's lemma
\begin{equation*}
\int_{\Gamma \backslash \mathbb{H}^2} |G_n^{\pm}(z,w)|^2\,d\mu(z) \le \liminf\limits_{K\to \infty} \int_{\Gamma \backslash \mathbb{H}^2}|G_{n,K}^{\pm}(z,w)|^2\,d\mu(z) \ll_{\Gamma} (1+n)^2. 
\end{equation*}
\end{proof}

\subsection{Applying the Borel-Cantelli Lemma} 

Fix $\eta>0$ and set $A_{n,\eta} = n^{\frac{3}{2}}(\log(n))^{\frac{1}{2}+\eta}$. Consider the set of points  
\begin{equation*}
E_{n,\eta}^{\pm}(w) = \{z \in \mathbb{H}^2 : G_n^{\pm}(z,w) > A_{n,\eta}\}. 
\end{equation*}
The Chebyshev inequality and Proposition \ref{nplus1} imply that 
\begin{equation*}
\begin{split}
\mu(E_{n,\eta}^{\pm}(w))  &\le A_{n,\eta}^{-2} \int_{\Gamma \backslash \mathbb{H}^2} |G_n^{\pm}(z)|^2\,d\mu(z) \\
&\ll_{\Gamma}  \frac{1}{n\log(n)^{1+2\eta}}. 
\end{split}
\end{equation*}
The summability of the series of the sequence $\{n^{-1}\log(n)^{-(1+2\eta)}\}_n$ implies that 
\begin{equation*}
\sum\limits_{n=1}^{\infty} \mu(E_{n,\eta}^+ \cup E_{n,\eta}^-) <\infty. 
\end{equation*}
Thus, the first Borel-Cantelli Lemma implies that 
\begin{equation*}
\mu(\limsup\limits_{n\to \infty} (E^+_{n,\eta}\cup E_{n,\eta}^-)) = 0. 
\end{equation*}
Define $B_{w,\eta}:= (\Gamma \backslash \mathbb{H}^2)\backslash \limsup\limits_{n\to \infty} (E^+_{n,\eta}\cup E_{n,\eta}^-)$. If $z\in Z_{w,\eta}$ then there exists an $n_0= n_0(z,w,\eta)$ such that 
\begin{equation*}
G_n^{\pm}(z,w) \le A_{n,\eta} \text{ for } n\ge n_0. 
\end{equation*}
For $R\in I_n$ this implies that 
\begin{equation*}
|S_n^{\pm}(R;z,w)|\le n^{\frac{3}{2}}(\log(n))^{\frac{1}{2}+\eta} \ll R^{\frac{3}{2}}(\log(R))^{\frac{1}{2}+\eta}
\end{equation*}
and as a result that 
\begin{equation*}
\left|\sum\limits_{t_j\geq 1} h^{\pm}_{R,\delta(R)}(t_j)\phi_j(z)\overline{\phi_j(w)}\right|= \left|e^{\frac{R} {2}}S_n^{\pm}(R;z,w)\right| \ll_{M,w,z,\eta} e^{\frac{R}{2}}R^{\frac{3}{2}}(\log(R))^{\frac{1}{2}+\eta}. 
\end{equation*}
This completes the proof of Theorem \ref{maximalestimate}. 

\subsection{Completing the proof of our main theorem} The proof of Theorem \ref{ourmainthm1} follows by taking, for $\eta=1/m$, the intersection
\begin{equation*}
B_w = \bigcap_{m=1}^{\infty} B_{w,1/m}.
\end{equation*}
Since $\mu(B_{w,1/m})=1$ for every $m\geq1$, it follows that $\mu(B_w)=1$. The proof of Theorem \ref{ourmainthm1} is complete.

\section{The second moment of the error term}

\subsection{An improved upper bound}

In this subsection we prove Theorem \ref{cramercloser}, which we now recall. 

\begin{theorem}\label{XloglogX}
Assume $\Gamma$ is cocompact. Given any $w\in \mathbb{H}^2$ there exists a subset $B_w \subset \Gamma \backslash \mathbb{H}^2$ of full hyperbolic measure such that for any $z\in B_w$ and every sufficiently large $X$:
\begin{equation*}
\frac{1}{X} \int_X^{2X} |E(x;z,w)|^2\,dx \ll_{\Gamma,z,w} X\log\log X. 
\end{equation*}
\end{theorem}


As in the work of Phillips and Rudnick \cite{PhillipsRudnick:1994}, let us consider the normalized error term given by
\begin{equation*}
e(R;z,w) = e^{-\frac{R}{2}} E(2\cosh R;z,w).
\end{equation*}
In order to deduce Theorem \ref{XloglogX}, it suffices to study the second moment 
\begin{equation*}
\int_{A}^{A+3} |e(R;z,w)|^2 dR
\end{equation*}
uniformly for $A \geq 1$. For $R,t\geq 1$ we will use the uniform bound \cite{Chamizo:1996a}
\begin{equation} \label{hrt}
h_R(t) = 2\sqrt{2\pi\sinh R} t^{-\frac{3}{2}}\cos \left(tR - \frac{3\pi}{4}\right) + O (e^{R/2}t^{-5/2}). 
\end{equation}
For the estimates below we can ignore the contribution of the error \eqref{hrt} completely since it contributes only bounded quantities. Define 
\begin{equation*}
F_T (R;z,w) := e^{-R/2} \, \sum\limits_{0<t_j\le T} h_R(t_j) \phi_j(z) \overline{\phi_j(w)},
\end{equation*}
and the tail
\begin{equation*}
G_T (R;z,w):= e(R;z,w) - F_T (R;z,w).
\end{equation*}
Clearly, for any $A\geq 1$ we get the upper bound
\begin{equation*}
\int _A^{A+3}|e(R;z,w)|^2\,dR \le 2\left(\int _A^{A+3}|F_T(R;z,w)|^2\,dR+\int _A^{A+3}|G_T(R;z,w)|^2\,dR\right).
\end{equation*}
We will bound separately the second moment of the bulk and the tail of the series. We first have the following estimate. 
\begin{lemma}\label{L1}
Uniformly for $R,A\geq 1$,
\begin{equation*}
\int _A^{A+3}|F_T(R;z,w)|^2\,dR \ll_{\Gamma} \log T.
\end{equation*}
\end{lemma}
\begin{proof}
Up to the easily controlled error term $O(e^{R/2}t^{-5/2})$ the quantity we need to control is given by 
\begin{equation*}
\sum\limits_{0<t_j\le T} a_j \cos(Rt_j-3\pi/4)  = \Re(e^{-3i\pi/4}S(R)),
\end{equation*}
where 
\begin{equation*}
a_j = t_j^{-3/2}\phi_j(z) \overline{\phi_j(w)} 
\end{equation*}
and
\begin{equation*}
S(R) = \sum\limits_{0<t_j \le T}a_je^{it_jR}. 
\end{equation*}
Since its norm is bounded above by $|S(R)|$, it thus suffices to show that 
\begin{equation*}
\int_A^{A+3} |S(R)|^2\,dR \ll_{\Gamma} \log T. 
\end{equation*}
Let us fix a nonnegative smooth compactly supported function $\chi$ such that $\chi =1$ on $[0,3]$ and $\supp \chi \subset [-1,4]$. Let $\chi_A(R) = \chi(R-A)$. Clearly, 
\begin{equation*}
\int_A^{A+3} |S(R)|^2\,dR \le \int_{\R} \chi_A(R)|S(R)|^2\,dR. 
\end{equation*}
Define 
\begin{equation*}
K_A(\xi) = \int _{\R} \chi_A(R) \, e^{i\xi R} \,dR. 
\end{equation*}
Clearly $|K_A(\xi)|\le \|\chi\|_{L^1}$. Performing two integrations by parts with no boundary terms, since $\chi$ is smooth and compactly supported, gives
\begin{equation*}
|K_A(\xi)| \le \frac{\|\chi''\|_{L^1}}{|\xi|^2},~\xi\neq 0.
\end{equation*}
Combining the two estimates we get 
\begin{equation*}
|K_A(\xi)| \le \frac{C_{\chi}}{(1+|\xi|)^2}. 
\end{equation*}
Using this estimate, we have 
\begin{equation}\label{B1}
\begin{split}
\int_{\R} \chi_A(R)|S(R)|^2\,dR &= \sum\limits_{i,j} a_i\overline{a_j} K_A(t_i-t_j) \\
&\le C_{\chi}\sum\limits_{i,j} \frac{|a_i||a_j|}{(1+|t_i-t_j|)^2}. 
\end{split}
\end{equation}
Now, let us define 
\begin{equation}\label{B2}
B_n = \sum\limits_{\substack{n\le t_j <n+1\\ t_j\le T}} \frac{|\phi_j(z) \overline{\phi_j(w)}|}{t_j^{3/2}} \ll_{\Gamma} n^{-1/2},
\end{equation}
where the inequality is a direct consequence of the local Weyl law on the unit interval $[n,n+1]$. A direct consequence of \eqref{B1} is the bound 
\begin{equation*}
\begin{split}
\int_A^{A+3} |S(R)|^2\,dR &\le C_{\chi}\sum\limits_{i,j} \frac{|a_i||a_j|}{(1+|t_i-t_j|)^2} \\
&\ll_{\Gamma} \sum\limits_{n,m\le T+1} \frac{B_nB_m}{(1+|n-m|)^2} \\
&\ll_{\Gamma} \sum\limits_{n\le T+1} B_n^2 \\
&\ll_{\Gamma} \log T,
\end{split}
\end{equation*}
where the last inequality follows from \eqref{B2}. 
\end{proof}
We now bound the second moment of the tail. 
\begin{lemma}\label{L2}
Uniformly for $R \geq 1$,
\begin{equation*}
\int_{\Gamma \backslash \mathbb{H}^2}|G_T (R;z,w)|^2\,d\mu(z) \ll_{\Gamma} T^{-1}. 
\end{equation*}
\end{lemma}

\begin{proof}
Orthogonality of the eigenfunctions and the local Weyl law over dyadic intervals immediately give 
\begin{equation*}
\begin{split}
\int_{\Gamma \backslash \mathbb{H}^2}|G_T (R;z,w)|^2\,d\mu(z)  &= e^{-R}\sum\limits_{T<t_j} h_R(t_j)^2|u_j(w)|^2\\
&\ll_{\Gamma} \sum\limits_{T<t_j} t_j^{-3}|u_j(w)|^2 \ll_{\Gamma} T^{-1}. 
\end{split}
\end{equation*}
\end{proof}

\begin{proof}[Proof of Theorem \ref{XloglogX}]
For $k\geq 1$ choose $L_k = (k+1)^2$ and define 
\begin{equation*}
Q_k(z,w) = \int_k^{k+3} |G_{L_k}(R;z,w)|^2\,dR.
\end{equation*}
It follows from Lemma \ref{L2} that 
\begin{equation*}
\sum\limits_k \int_{\Gamma \backslash \mathbb{H}^2} Q_k (z,w)  \,d\mu(z) <\infty.
\end{equation*}
Applying Tonelli's theorem to change the order of integration we get 
\begin{equation*}
\int_{\Gamma \backslash \mathbb{H}^2} \sum\limits_k Q_k(z,w) \,d\mu(z) <\infty,
\end{equation*}
which implies that there exists a full measure subset $B_w\subset \Gamma \backslash \mathbb{H}^2$ of points $z$ for which 
\begin{equation*}
Q(z,w) = \sum\limits_k Q_k(z,w) <\infty. 
\end{equation*}
This, together with Lemma \ref{L1} imply that simultaneously for all $k$ and for every $z \in B_w\subset \Gamma \backslash \mathbb{H}^2$ 
\begin{equation*}
\int_k^{k+3} |e(R;z,w)|^2\,dR \ll \log(k) + Q(z,w) \ll_{\Gamma,z,w} \log(k). 
\end{equation*}
For large $X$ the interval $[X,2X]$ is contained in the radius (in the $R$-variable) interval $[k,k+3]$, which gives 
\begin{equation*}
\frac{1}{X} \int_X^{2X} |E(x;z,w)|^2\,dx \ll_{\Gamma,z,w} X\int_k^{k+3} |e(R;z,w)|^2\,dR \ll_{\Gamma,z,w} X\log\log(X).
\end{equation*}
\end{proof}

\subsection{Failure of the conjectured lower bound} 

We proceed to the proof of Theorem \ref{cramerfailure}, which we recall. 

\begin{theorem}
For every $z \in \Gamma \backslash \mathbb{H}^2$ it holds that 
\begin{equation*}
\limsup\limits_{X\to \infty} \frac{1}{X^2}\int_X^{2X} |E(x;z,z)|^2\,dx  = \infty. 
\end{equation*}
\end{theorem}

\begin{proof}
Let $a_j = |\phi_j(z)|^2$ and define the tempered distribution on $\R$ 
\begin{equation*}
F(s) = \sum\limits_{1\le t_j} a_j t_j^{-3/2}\cos(t_js - 3\pi/4). 
\end{equation*}
It holds that 
\begin{equation*}
e(R;z,z) = 2\sqrt{\pi(1-e^{-2R})}F(R) +B(R), \ \ |B(R)|\ll_{\Gamma,z}1. 
\end{equation*}
By way of contradiction, assume that the limit supremum in the theorem is finite. This implies a uniform bound on the $L^2(R)$-norm of $e(R;z,z)$ on intervals of finite length and thus uniform bounds
\begin{equation}\label{lL2b}
\int_{-3}^3|F_R(s)|^2\,ds \le C,
\end{equation}
where $F_R(s) = F(R+s)$. 

Simultaneous approximation of finitely many frequencies $t_j$ followed by a diagonal argument gives a sequence $\{R_n\}_n$ with $R_n\to \infty$ and $e^{it_jR_n}\to 1$. Thus distributionally $F_{R_n}(s)\to F(s)$. The uniform local $L^2$-bounds \eqref{lL2b} and weak compactness implies that $F\in L^2(-3,3)$. 

We show this is impossible. Choose a nonzero, nonnegative $\psi\in C^{\infty}_c((-2,-1))$ and write $\psi_\eps(s) = \eps^{-1/2}\psi(s/\eps)$. Let 
\begin{equation*}
K(u) = \int_{\mathbb{R}} \psi(u)\cos(uv-3\pi/4)\,dv, 
\end{equation*}
and $A(T) = \sum\limits_{1\le t_j\le T} a_j = c_z T^2 +O_z (T)$. Stieltjes partial summation gives 
\begin{equation*}
\begin{split}
\langle F,\psi_{\eps}\rangle &= 4 \pi c_z \sqrt{\eps} \int_{1^-}^{\infty} t^{-3/2}K(\eps t)\,dA(t) \\
&= 2 c_z \int_{\eps}^{\infty} u^{-1/2}K(u)\,du + O_z (\sqrt{\eps}).
\end{split}
\end{equation*}
Indeed, the $O_z (T)$ remainder term in $A(T)$ contributes at most a constant times 
\[
\sqrt{\eps}\left(1+\int_1^{\infty}t^{-3/2}|K(\eps t)|\,dt+\eps\int _1^{\infty}t^{-1/2}|K'(\eps t)|\,dt\right) = O(\sqrt{\eps}).
\]
Using Abel damping to interchange integrals, the identity 
\[
\int_0^{\infty} u^{-1/2}\cos(uv -3\pi/3)\,du = -\sqrt{\pi}(-v)^{-1/2},
\]
for $v<0$ shows that 
\[
\lim\limits_{\eps\to 0} \langle F,\psi_{\eps}\rangle = - 2 c_z \sqrt{\pi} \int_{-2}^{-1}\frac{\psi(v)}{\sqrt{-v}}\,dv \neq 0. 
\]
On the other hand, if $F\in L^2(-3,3)$ the by Cauchy-Schwarz 
\[
|\langle F,\psi_{\eps}\rangle|\le \|F\|_{L^2(-2\eps,-\eps)}\|\psi\|_2 \to 0
\]
which is a contradiction. 
\end{proof}

\section{Higher dimensional hyperbolic manifolds}

In this last section we prove Theorem \ref{higherdim}, recording only the changes needed in Sections $2$ and $3$. Let  $n \geq 3$ and let $\mathbb{H}^n$ denote the $n$-dimensional hyperbolic space. Let also $\Gamma \subset \hbox{SO}^+(1,n)$ be a discrete cocompact group of isometries acting on $\mathbb{H}^n$. The quotient $\GmodH^n$ is a $n$-dimensional geometrically finite orbifold of constant curvature $K=-1$. 

\subsection{Spectral background}

The eigenvalues of the (positive) Laplace--Beltrami operator are ordered by $0 =\lambda_0 \leq \lambda_1 \leq ... \leq \lambda_j \to \infty$, and we write 
\begin{equation*}
\lambda_j = s_j (n-1-s_j) =  \left(\frac{n-1}{2}\right)^2 + t_j^2,
\end{equation*}
with $s_j = (n-1)/2 + it_j$. We retain the notation $X = 2 \cosh R$.

\subsection{The spherical transform}

Assume we take a characteristic kernel $k$ of the unit ball defined similarly with \eqref{kernel}. The spherical transform of $k$ is given by
\begin{equation*}
h_R(t) =  \frac{(2\pi)^{\frac{n-1}{2}}}{\Gamma \left( \frac{n+1}{2}\right)} \int_{-R}^{R} \left(\cosh R - \cosh r\right)^{\frac{n-1}{2}} e^{itr} \, dr.
\end{equation*}
We consider again the automorphic kernel 
\begin{equation*}
K_X(z,w) =  K(z,w)=\sum_{\g\in \G}k(u(z,\gamma w)),
\end{equation*}
which satisfies $K(z,w) = N(X;z, w) $ (here by $N(X;z,w)$ we mean the new $n$-dimensional counting function). For the same reasons as in $2$-dimensions the pre-trace formula cannot be applied directly to $K(z,w)$ as the spherical transform of $k$ decays slowly (even slower in terms of the dimension $n$). We thus use again smooth approximations to the automorphic kernel $K(z,w)$.

We need the asymptotic behavior of the $n$-dimensional spherical transform $h_R(t)$ of the characteristic function $k$, which is given in terms of the associated Legendre function by
\begin{equation} \label{eq:hRdefhigher}
  h_R(t)= 2\pi (\sinh R)^{n/2} P_{-\frac{1}{2}+it}^{-n/2}(\cosh R),
\end{equation}
which is the higher dimensional analogue of \eqref{eq:hRdef} (see \cite{katsivelos, PhillipsRudnick:1994} and \cite[eq.~]{GradshteynRyzhik:2007}). Analyzing the properties of $h_R(t)$ using \eqref{eq:hRdefhigher} we obtain the following Lemma. 

 \begin{lemma} \label{sphericalpropertieshigher}
The spherical transform $h_R(t)$ is entire in $t \in \mathbb{C}$ and satisfies:
\\\\
(a) for real $|t| \geq 1, R \geq 1$ we have
\begin{equation*} 
  h_R(t)=O\left(\frac{e^{(n-1)R/2}}{(1+{\abs{t})}^{(n+1)/2}}\right), \ \ \ \ 
  \partial_R \, h_R(t) = \left( \frac{e^{(n-1)R/2}}{(1+{\abs{t})}^{(n-1)/2}}\right),
\end{equation*}
(b) for real $|t|\leq 1$ we also have
\begin{equation*}
h_R(t) =  O \left((1+R) \, e^{(n-1)R/2)}\right),    
\end{equation*}
(c) for real $t \neq 0$ and $0< \delta < 1/4$ we have 
\begin{equation*} 
m_{\delta}(t):= \frac{h_{\delta}(t)}{\hbox{vol}(B_{\delta})} = O\left(\frac{1}{(1+{\abs{\delta t})}^{(n+1)/2}}\right),
\end{equation*}
where by $B_{\delta}$ we denote the ball of radius $\delta$,
\\
(d) for $0 \leq s_j \leq (n-1)/2$ we have
\begin{equation*}
 \mu_{\delta}     \left(s_j\right) = 1 +O_n(1+\delta^2).
\end{equation*}
\end{lemma}

\subsection{The main term and the error}

In contrast to dimension $2$, in higher dimensions the main term $M(X;z,w)$ of the counting function becomes more complicated. In fact, the main term takes the form
\begin{eqnarray*} 
M_\Gamma(X;z,w)=\sum_{\frac{n-1}{2} \leq s_j \leq n-1} \pi^{\frac{n-1}{2}} \frac{\Gamma \left(s_j-\frac{n-1}{2}\right)}{\Gamma(s_j+1)} \phi_j(z) \overline{\phi_j(w)} X^{s_j} + A_{\Gamma}(X; z,w),
\end{eqnarray*}
see \cite{PhillipsRudnick:1994}. Here the secondary main term $A_\Gamma(X;z,w)$ is a finite sum
\begin{eqnarray*}
A_{\Gamma}(X; z,w)= \sum_{\frac{n-1}{2} \leq s_j \leq n-1} a(X; s_j ,z,w), 
\end{eqnarray*}
where for every $s_j > (n-1)/2$ the quantity $a(X;s_j,z,w)$ contains a contribution coming from the eigenvalue $\lambda_j < (n-1)^2/4$ (if this eigenvalue exists). Moreover, it satisfies $a(X;s_j,z,w) \ll X^{s_j-2}$ \cite[eq.~(4.3)]{PhillipsRudnick:1994}. The term $a(X;(n-1)/2,z,w)$ includes the total contribution of the eigenvalue $\lambda_j = (n-1)^2/4$ and satisfies 
\begin{eqnarray*} \label{smallcontribution}
a(X;(n-1)/2,z,w) = O(X^{(n-1)/2+\epsilon}).
\end{eqnarray*}
Overall, we conclude that $ A_{\G} (X;z,w) = O (X^{n-3})$.
We refer to \cite[sect. 4]{PhillipsRudnick:1994} for a detailed explanation of this term.

\subsection{The proof of Theorem \ref{higherdim}}

We proceed to the proof, which works similarly as in $2$ dimensions. We convolute the kernel $K_X$ with the normalized ball kernel $k_{\delta} =  (\hbox{vol}(B_{\delta}))^{-1}\mathbf{1}_{B_{\delta}}$ and then we use Lemma \ref{sphericalpropertieshigher}. On each interval $I_{\ell} = [\ell,\ell+1]$ with $\ell \geq 2$, we choose $\delta (R) = \delta_{\ell} = e^{-{(n-1)(\ell+1)/2}}$.

Let us set the normalized spherical transform
\begin{equation*}
c_{\ell,t}^{\pm} (R) = e^{-(n-1)R/2} \, h_{R \pm \delta_{\ell}}(t) \, m_{\delta} (t)    
\end{equation*}
and the weights
\begin{equation*}
   Q_{\ell,T} = (1+\delta_{\ell} T)^{-\frac{n+1}{2}}. 
\end{equation*}
Since $\delta_{\ell}$ is constant on intervals $I_{\ell}$, we easily get the following estimate.
\begin{lemma}\label{highercpmdcpm}
The following estimates hold uniformly for $t\geq 1$ and for $R \in I_{\ell}$:
\begin{equation*}
\begin{split}
|c_{\ell,t}^{\pm}(R)| &\ll t^{-\frac{n+1}{2}}Q_{\ell,t},\\\\
|\partial_R c^{\pm}_{\ell,t}(R)| &\ll t^{-\frac{n-1}{2}}Q_{\ell,t}. 
\end{split}
\end{equation*}
\end{lemma}
As in \eqref{Bnpm} we define the dyadic sums
\begin{equation} \label{higherBnpm}
B_{\ell}^{\pm}(T,R; z,w) := \sum\limits_{T\le t_j<2T} c_{\ell,t_j}^{\pm} (R) \,  \phi_j(z) \, \overline {\phi_j(w)}. 
\end{equation}
We have the following $L^2$-bounds for these convoluted dyadic sums. Using the local Weyl law in $n$-dimensions
\begin{equation*}
   \sum_{t_j \leq T} |\phi_j(w)|^2 \sim c_w T^n, \ \ \text{for some constant} \ \ c_w >0 
\end{equation*}
and orthogonality, we deduce an analogue of Proposition \ref{L2bounds}.

\begin{proposition}\label{higherL2bounds}
For every $\ell\geq 2$, every dyadic $T\geq 1$ and every $R\in I_{\ell}$ the following bounds hold: 
\begin{equation*}
\begin{split}
\|B_{\ell}^{\pm}(T,R; w)\|_2^2:=& \|B_{\ell}^{\pm}(T,R; w)\|_{L^2(\Gamma \backslash \mathbb{H}^n)}^2 \ll_{\Gamma} T^{-1}Q_{\ell,T}^2, \\\\
\|\partial_R B_{\ell}^{\pm}(T,R; w)\|_2^2:=& \|\partial_R B_{\ell}^{\pm}(T,R; w)\|_{L^2(\Gamma \backslash \mathbb{H}^n)}^2 \ll_{\Gamma} T  Q_{\ell,T}^2.
\end{split}
\end{equation*}
\end{proposition}
We can thus applu the same technique as in the $2$-dimensional case. We have
\begin{equation*}
\sum\limits_{k=0}^{\infty} Q_{\ell,2^k} \ll_{\Gamma} 1+\ell.
\end{equation*}
and the dyadic sum $T^{(n-1)/2} Q_{\ell,T}$ is summable over dyadic $T=2^k$. The contribution of the small $|t_j| \leq 1$ 
is trivially bounded by $(1+R) e^{(n-1)R/2} \ll X^{(n-1)/2} \log X$ (since they give a finite sum). The loss of the convoluted main term contributes 
\begin{equation*}
O_{\Gamma,z,w} \left(\delta_{\ell} e^{(n-1)R} + e^{(n-1)R/2}\right) = O_{\Gamma,z,w} \left(e^{(n-1)R/2}\right).
\end{equation*}
Finally, the first Borel--Cantelli Lemma, after picking $\eta = 1/m$ and taking the intersection, implies Theorem \ref{higherdim}.

\bibliographystyle{amsplain}

\end{document}